\documentclass{amsart}
\usepackage{amsmath,amsfonts,amsthm,amssymb,indentfirst,epic,url,graphics,needspace}

\newtheorem{theorem}{Theorem}
\newtheorem{lemma}[theorem]{Lemma}

\newtheorem{proposition}[theorem]{Proposition}

\newtheorem{question}[theorem]{Question}

\numberwithin{equation}{section}
\numberwithin{theorem}{section}

\begin{document}

\begin{center}
\texttt{Comments, corrections,
and related references welcomed, as always!}\\[.5em]
{\TeX}ed \today
\vspace{2em}
\end{center}

\title%
[Infinite and finite direct sums]%
{An elementary result on \\
infinite and finite direct sums of modules}
\thanks{%
Archived at \url{http://arxiv.org/abs/2208.06511}\,.
Also readable at \url{http://math.berkeley.edu/~gbergman/papers/}.
The latter version may be revised more frequently than the former.
After publication, any updates, errata, related references,
etc.\ found will be recorded at
\url{http://math.berkeley.edu/~gbergman/papers/}.\\
}


\author{George M.\ Bergman}
\address{Department of Mathematics\\
University of California\\
Berkeley, CA 94720-3840, USA}
\email{gbergman@math.berkeley.edu}

\begin{abstract}
Let $R$ be a ring, and consider a left $\!R\!$-module
given with two (generally infinite) direct sum decompositions,
$A\oplus(\bigoplus_{i\in I} C_i)=M=B\oplus(\bigoplus_{j\in J} D_j),$
such that the submodules $A$ and $B$ and the $D_j$ are each finitely
generated.
We show that there then exist {\em finite} subsets $I_0\subseteq I,$
$J_0\subseteq J,$ and a direct summand
$Y\subseteq \bigoplus_{i\in I_0} C_i,$ such that
$A \oplus Y \ =\ B \oplus(\bigoplus_{j\in J_0} D_j).$

We then note some ways that this result can and cannot be generalized,
and pose some related questions.
\end{abstract}
\maketitle

\section{The result}\label{S.main}

Throughout this note, $R$ will be an associative unital ring,
and `module' will mean unital left $\!R\!$-module.
An easy observation is

\begin{lemma}[$\approx$\cite{Fac}, Lemma~2.1, p.33]\label{L.oplus}
Given a module with a direct sum decomposition $S\oplus T,$ and a
submodule $P$ thereof containing $S,$ one has $P=S\oplus (T\cap P).$
\qed
\end{lemma}

The next proposition will be key to the
proof of our main result, Theorem~\ref{T.main}.
I am grateful to Pace Nielsen
and Mauricio Medina-B\'arcenas for this version of
the proposition, stronger than the one in an earlier draft of
this note, and leading to a stronger version of Theorem~\ref{T.main},
answering a question I had previously posed.

\begin{proposition}[P.\,Nielsen, M.\,Medina-B\'arcenas]\label{P.4term}
Suppose we are given a module with a direct
sum decomposition $B\oplus V,$
and a submodule which admits a direct sum
decomposition $A\oplus X,$ such that
\begin{equation}\begin{minipage}[c]{35pc}\label{d.Bsseteq}
$A\ \subseteq\ B\ \subseteq\ A \oplus X\ \subseteq\ B\oplus V$
\end{minipage}\end{equation}
\textup{(}with no assumption on whether $X\subseteq V).$

Then we have a direct sum decomposition
\begin{equation}\begin{minipage}[c]{35pc}\label{d.ABA+}
$B\ =\ A\oplus(X\cap B),$
\end{minipage}\end{equation}
in which the summand $X\cap B$ is also a direct summand in $X;$
precisely:
\begin{equation}\begin{minipage}[c]{35pc}\label{d.sums}
$X\ =\ (X\cap B)\ \oplus\ (X\cap(A\oplus V)).$
\end{minipage}\end{equation}
\end{proposition}

\begin{proof}
We apply Lemma~\ref{L.oplus}, on the one hand
to the first three terms of~\eqref{d.Bsseteq}, and on
the other hand to the last three terms, in each case
taking the leftmost term for $S,$ the direct summand attached
to it in the rightmost term for $T,$ and the middle term for $P.$
The application to the first three terms of~\eqref{d.Bsseteq}
gives~\eqref{d.ABA+}, while the application to the last three gives
\begin{equation}\begin{minipage}[c]{35pc}\label{d.BA+B+}
$A\oplus X\ =\ B\oplus (V\cap (A\oplus X)).$
\end{minipage}\end{equation}
Substituting~\eqref{d.ABA+} into the right-hand side
of~\eqref{d.BA+B+} gives
\begin{equation}\begin{minipage}[c]{35pc}\label{d.messy}
$A\oplus X\ =\ A \oplus (X\cap B) \oplus (V\cap (A\oplus X)).$
\end{minipage}\end{equation}
We now make a third application of Lemma~\ref{L.oplus}, this time
taking for $S$ the middle summand $X\cap B$ on the right-hand side
of~\eqref{d.messy}, for $T$ the sum of the other two
terms on that side, and for $P$ the module $X,$ which, in view
of the left-hand side of~\eqref{d.messy}, is contained
in the right-hand side, and which clearly contains our chosen $S.$
The result is
\begin{equation}\begin{minipage}[c]{35pc}\label{d.pre-sums}
$X\ =\ (X\cap B)\ \oplus\ ((A\oplus(V\cap(A\oplus X)))\cap X).$
\end{minipage}\end{equation}
To get~\eqref{d.sums}, it will suffice to show that the terms
at which~\eqref{d.sums} and~\eqref{d.pre-sums} differ are the same;
i.e., that:
\begin{equation}\begin{minipage}[c]{35pc}\label{i.e., that}
$X\cap(A\oplus V)\ =\ (A\oplus(V\cap(A\oplus X)))\cap X.$
\end{minipage}\end{equation}
Here $\supseteq$ is immediate, since the right-hand side can,
essentially, be obtained from the left by shrinking
$V$ to $V\cap(A\oplus X).$
To get the reverse inclusion, consider any element of the
left-hand side, and let us write it
\begin{equation}\begin{minipage}[c]{35pc}\label{d.a+v}
$a + v\ =\ x,$ where $a\in A,\ v\in V,\ x\in X.$
\end{minipage}\end{equation}
If we rewrite the above equation as
\begin{equation}\begin{minipage}[c]{35pc}\label{d.-a+x}
$v\ =\ -a+x,$
\end{minipage}\end{equation}
we see that $v$ belongs to $V\cap(A\oplus X),$ hence $a+v$ belongs
to $A\oplus (V\cap(A\oplus X)),$ and since $a+v=x,$ this element
in fact belongs to $(A\oplus (V\cap(A\oplus X)))\cap X,$ as required.
\end{proof}

We can now prove

\begin{theorem}\label{T.main}
Suppose $M$ is a module having two
\textup{(}generally infinite\textup{)} direct sum decompositions,
\begin{equation}\begin{minipage}[c]{35pc}\label{d.M}
$A \oplus(\bigoplus_{i\in I} C_i) \ =\ M\ =\ %
B \oplus(\bigoplus_{j\in J} D_j),$
\end{minipage}\end{equation}
where the summands $A$ and $B,$ and each of the $D_j,$
are finitely generated.

Then there exist finite subsets $I_0\subseteq I$ and $J_0\subseteq J,$
and a direct summand $Y$ in $\bigoplus_{i\in I_0} C_i,$ such that
\begin{equation}\begin{minipage}[c]{35pc}\label{d.fin_sums}
$A \oplus Y\ =\ B \oplus(\bigoplus_{j\in J_0} D_j).$
\end{minipage}\end{equation}
\end{theorem}

\begin{proof}
Since $A$ is finitely generated, it is contained
in the sum of a finite subset of the modules on the right-hand
side of~\eqref{d.M}, which we may take to include $B.$
That is,
\begin{equation}\begin{minipage}[c]{35pc}\label{d.B_sset}
$A\ \subseteq\ B \oplus(\bigoplus_{j\in J_0} D_j),$~
where $J_0\subseteq J$ is finite.
\end{minipage}\end{equation}

Since $B$ and the $D_j$ are all finitely generated we
can, again in view of~\eqref{d.M},
choose finitely many of the $C_i$ such that
\begin{equation}\begin{minipage}[c]{35pc}\label{d.A_oplus_sset}
$B \oplus(\bigoplus_{j\in J_0} D_j)\ \subseteq\ %
A \oplus(\bigoplus_{i\in I_0} C_i),$~ where $I_0\subseteq I$ is finite.
\end{minipage}\end{equation}

Bringing together
the inclusions~\eqref{d.B_sset} and~\eqref{d.A_oplus_sset},
and putting $\subseteq B\oplus(\bigoplus_{j\in J} D_j)$
on the far right (in view of~\eqref{d.M}),
we get a chain like~\eqref{d.Bsseteq},
with $B \oplus(\bigoplus_{j\in J_0} D_j)$
in the role of the $B$ of~\eqref{d.Bsseteq},
$\bigoplus_{i\in I_0} C_i$ in the role of $X,$
and $\bigoplus_{j\in J\setminus J_0} D_j$ in the role of $V;$
so we can apply Proposition~\ref{P.4term}.
The conclusion~\eqref{d.ABA+} of that proposition (with sides
reversed) gives us a decomposition~\eqref{d.fin_sums},
with $Y = (\bigoplus_{i\in I_0} C_i)\cap
(B \oplus(\bigoplus_{j\in J_0} D_j)),$
and the line after~\eqref{d.ABA+} tells us that this $Y$ is indeed
a direct summand in $X = \bigoplus_{i\in I_0} C_i$
(with a complement that can be described using~\eqref{d.sums}, though we
have not included that description in the statement of the present
theorem).
\end{proof}

Note that if, rather than being given, as above, a module $M$
with two internal direct-sum decompositions, we are given an
isomorphism between two external direct sums of families of modules,
\begin{equation}\begin{minipage}[c]{35pc}\label{d.inf_sums}
$A \oplus(\bigoplus_{i\in I} C_i)\ \cong\ %
B \oplus(\bigoplus_{j\in J} D_j),$
\end{minipage}\end{equation}
with $A,$ $B$ and the $D_j$ again all finitely generated,
then we can apply the above theorem by taking for $M$ the
left-hand side of~\eqref{d.inf_sums}, and regarding
the inverse images in $M,$ under the isomorphism~\eqref{d.inf_sums},
of the summands on the right-hand side as submodules of $M,$
yielding a second direct-sum decomposition of $M.$
Theorem~\ref{T.main}, applied to that pair of decompositions, then tells
us that for some finite subsets $I_0\subseteq I$ and $J_0\subseteq J,$
and some direct summand $Y$ of $\bigoplus_{i\in I_0} C_i,$ we have
\begin{equation}\begin{minipage}[c]{35pc}\label{d.fin_sums_iso}
$A \oplus Y\ \cong\ B \oplus(\bigoplus_{j\in J_0} D_j).$
\end{minipage}\end{equation}

\section{Remarks}\label{S.remarks}

Theorem~\ref{T.main} was motivated by a question posed by
Z.\,Nazemian (personal communication),
roughly: if $Q$ and $P$ are finitely generated
projective modules over a ring $R,$ such that every direct summand of
a module $P^a$ for a natural number $a$ is isomorphic to $P^b$ for a
natural number $b,$ and if $Q^m\oplus(\bigoplus_{i=1}^\infty P)\cong
Q^n\oplus(\bigoplus_{i=1}^\infty P)$ for some $m$ and $n,$ then must
$Q^m\oplus P^k\cong Q^n\oplus P^l$ for some finite $k$ and $l$?
I obtained an affirmative answer, and eventually generalized
the proof to give a result which, with help from Pace Nielsen,
became Theorem~\ref{T.main}.
(Nazemian's question had some further hypotheses which I
don't state above.
If the answer had been a counterexample,
then, of course, it would have been important to find one
that satisfied those restrictions; but since a positive
result was found, the fewer hypotheses the better.)

I have no cases of Theorem~\ref{T.main} in mind in which
the $C_i$ are not finitely generated.
They are not so assumed simply because the
proof does not need such an assumption.

There are several papers in the literature proving relationships
among expressions of a module as finite and/or infinite direct sums;
cf.\ \cite{A+G+O'M+P}, \cite{Fac}, \cite{Fac+Levy}, \cite{Kap}.
All of these, however, have strong assumptions either
on the base-ring (e.g., in\ \cite{A+G+O'M+P}, that it is what the
authors call a {\em separative exchange ring}) or on
the summand modules (e.g., in \cite{Fac+Levy},
that they have semilocal endomorphism rings).
I am not aware of any known results with hypotheses as weak
as those of Theorem~\ref{T.main}.

The phrase ``An elementary result'' in the title of this
note is half tongue-in-cheek.
Theorem~\ref{T.main} is indeed elementary in
that it does not call on any deep definitions or results,
and its statement and proof are not long.
But finding these involved much frustrating trial
and error, and I still don't have good intuition about them.
In the next section we will note, inter alia, variations on that
theorem involving some less elementary hypotheses.

\section{Some ways our result can and cannot be generalized, and some questions}\label{S.questions}

In the finite relation~\eqref{d.fin_sums} that we obtain in
Theorem~\ref{T.main} from the generally infinite relation~\eqref{d.M},
the summands $A,$ $B$ and $D_j$ $(j\in J_0)$ are modules
occurring in that original relation, but~$Y$ is a
{\em direct summand} in a sum of such modules.
Can we get a similar result in which {\em all} the summands are (at
least up to isomorphism) terms from the original relation?

An easy example shows that we cannot.
Let $R$ be a field $k,$ let $I$ and $J$ both be countably infinite
sets, let each $C_i$ and each $D_j$ be a $\!k\!$-vector-space
of dimension~$2,$ and let $A$ and $B$ be finite-dimensional
$\!k\!$-vector-spaces, one of odd and the other of even dimension.
Note that both sides of~\eqref{d.inf_sums} are vector spaces
of countably infinite dimension, so~\eqref{d.inf_sums} holds;
and this isomorphism can, as noted, be turned into a
relation~\eqref{d.M}.
But if we form the direct sum of $A$ with
some finite sum of the $C_i$ and $D_j,$
and likewise of $B$ with some finite sum of the $C_i$ and $D_j,$
one of these sums will be odd-dimensional and the other
even-dimensional;
so no variant of~\eqref{d.fin_sums} or~\eqref{d.fin_sums_iso}
without a term like $Y$ can hold.

A couple of people have pointed out to me that in Theorem~\ref{T.main},
``finitely generated'' can be everywhere weakened
to the condition called ``dually slender'' in~\cite{E+G+T} and
``small'' in~\cite[{\S}2.9]{Fac}.
A module $A$ is so called~if
\begin{equation}\begin{minipage}[c]{35pc}\label{d.fin_subsum}
Every homomorphism of $A$ into an infinite direct
sum of modules has image in a finite subsum,
\end{minipage}\end{equation}
equivalently, if
\begin{equation}\begin{minipage}[c]{35pc}\label{d.sum_to_A}
$A$ is not the union of any countable strictly increasing
chain of proper submodules,\linebreak\hspace{5pc}
$A_0\subset A_1 \subset \dots \subset A_i\subset\dots.$
\end{minipage}\end{equation}
(To see this equivalence, first
assume~\eqref{d.fin_subsum}, and note that if we had
$A_0\subset A_1 \subset \dots$ with union $A,$
the induced map $A\to\bigoplus_{i\in\omega} A/A_i$ would
give a contradiction to that assumption.
Conversely, assuming~\eqref{d.sum_to_A}, suppose
$A$ had a homomorphism into a direct sum $\bigoplus_I E_i,$
which did not land in a finite subsum.
Then the induced homomorphism into some countable subsum
$\bigoplus_{n\in\omega} E_{i_n}$ $(i_0, i_1, \dots$ distinct
elements of $I)$ would have the same property; and we see that
the kernels $A_n$ of the induced homomorphisms
$A\to\bigoplus_{i>n} E_i$ $(n\in\omega)$
would contradict~\eqref{d.sum_to_A}.)

In another direction, it is not hard to see that one can,
in the statement and proof of Theorem~\ref{T.main},
replace modules with {\em torsion-free} abelian groups,
and {\em finitely generated} modules with such groups having
{\em finite rank}
(maximal number of linearly independent elements).
More generally, one has the obvious analogous
result with torsion-free modules over any commutative integral domain
in place of torsion-free abelian groups.

One might try to embrace all of these versions in a statement
about direct sums in an appropriate sort of additive category; note that
this would not in general be an abelian category, since the category
of torsion-free abelian groups does not have cokernels.
Rather than trying at this time to give a ``definitive''
generalization of Theorem~\ref{T.main}, I hope that
readers will be able to call on the proof, or provide some
modification thereof, in cases they want to use.

In a different direction, I wonder about

\begin{question}\label{Q.prod}
Can one get a result like Theorem~\ref{T.main}, but
with direct product modules $A\times\prod_{i\in I} C_i$ and
$B\times\prod_{j\in J} D_j$ in place of the direct sums
$A\oplus\bigoplus_{i\in I} C_i$ and $B\oplus\bigoplus_{j\in J} D_j$?
\end{question}

(The replacement of the {\em binary} operators ``$\oplus$'' with
``$\times$'' above is, of course, just a formality;
it is for the infinite families that direct sums and
direct products differ.)

A difficulty is that, unlike the situation for infinite
direct sums, a finitely generated submodule of an
infinite direct product need not lie in a finite subproduct.
(E.g., in an infinite product of $\!1\!$-dimensional
vector spaces over a field, consider the $\!1\!$-dimensional subspace
generated by an element with infinite support.)

A quasi-cure for that problem is to replace the condition
of finite generation by that of finitely {\em cogeneration}.
Recall that a module is said to be `finitely cogenerated' if
every family of submodules thereof having zero intersection has a
finite subfamily with zero intersection~\cite{Wiki_fcog}.
If a finitely cogenerated module is embedded
in a possibly infinite direct product of modules, we see that
its projection onto some finite subproduct thereof will indeed
be an embedding.
Using this, given an isomorphism
\begin{equation}\begin{minipage}[c]{35pc}\label{d.inf_prods}
$
A\times(\prod_{i\in I} C_i)\ =\ M\ \cong\ B\times(\prod_{j\in J} D_j)
$
\end{minipage}\end{equation}
where $A,$ $B$ and the $D_i$ are finitely cogenerated,
we can obtain a chain of embeddings
\begin{equation}\begin{minipage}[c]{35pc}\label{d.prod_chain}
$A\ \hookrightarrow\ B \times(\prod_{j\in J_0} D_j)\ %
\hookrightarrow\ A \times(\prod_{i\in I_0} C_i)\ %
\hookrightarrow\ B \times(\prod_{j\in J} D_j),$
\end{minipage}\end{equation}
where $I_0$ and $J_0$ are finite subsets of $I$ and $J.$

But so far as I can see, we cannot say in~\eqref{d.prod_chain},
as we could in the proof of Theorem~\ref{T.main},
that the composite of two successive embeddings
gives the inclusion associated with the indicated
direct product decompositions.
If we could, then since finite direct products of modules
are the same as finite direct sums, we would be able to
apply Proposition~\ref{P.4term} as before.

Actually, a formal dual of the approach of Theorem~\ref{T.main}
should involve a chain of {\em surjections} going in the opposite
direction to the embeddings of~\eqref{d.prod_chain}.

Mauricio Medina-B\'arcenas~\cite{MMB} has formulated a result
which dualizes in a different way the condition of
being a direct sum; whether his result is
relevant to Question~\ref{Q.prod} is not clear.

In any case, I don't at
present see a way to answer Question~\ref{Q.prod}.

A final question, about which I have not thought hard, is

\begin{question}\label{Q.non_mod}
For what sorts of structures other than modules
do analogs of Theorem~\ref{T.main} hold?
\end{question}

One finds that the development of that theorem works
for not-necessarily-abelian groups (with
`direct sum' taken, as for modules,
to mean the subgroup of the direct product consisting of
tuples with only finitely many non-identity terms).
I don't know whether there are versions of the result that allow, say,
semidirect products in place of direct products.
The paper \cite{Cr+Jo} obtains results on
direct decompositions of more general algebraic structures,
but, as in the other publications cited, a strong hypothesis
(in this case called the {\em exchange property}) is assumed.

\section{Acknowledgements}\label{S.ackn}

I am indebted to Brian Davey, Jonathan Farley, Laszlo Fuchs,
Pace Nielsen, Mauricio Medina-B\'arcenas, Kevin O'Meara, and J.-P.~Serre
for references to related literature and for other helpful comments
on this note.


\begin{thebibliography}{00}

\bibitem{A+G+O'M+P} P.\,Ara, K.\,R.\,Goodearl, K.\,C.\,O'Meara and
E.\,Pardo,
{\em Separative cancellation for projective modules over exchange
rings}, Israel J.\ Math.\ {\bf 105} (1998) 105-137.
(As of Jan.~7, 2023, on the Springer webpage for the Abstract of
\cite{A+G+O'M+P}, all occurrences of the
symbol $\curlyeqsucc$ should be $\cong.)$
MR1639739

\bibitem{Cr+Jo} Peter Crawley and Bjarni J\'{o}nsson,
{\em Refinements for infinite direct decompositions of algebraic
systems}, Pacific J.\ Math.\ {\bf 14} (1964) 797-855.
MR0169806

\bibitem{E+G+T} Paul C.\,Eklof, K.\,R.\,Goodearl and Jan Trlifaj,
{\em Dually slender modules and steady rings,}
Forum Math. {\bf 9} (1997) 61--74.
MR1426454

\bibitem{Fac} Alberto Facchini,
{\em Module theory, Endomorphism rings and direct
sum decompositions in some classes of modules},
Progress in Mathematics, 167.  Birkh\"{a}user Verlag, Basel, 1998.
xiv+285~pp.  ISBN: 3-7643-5908-0.  MR1634015.
Reprinted 2012 under Modern Birkh\"{a}user Classics.
ISBN: 978-3-0348-0302-1; 978-3-0348-0303-8.  MR3025306.

\bibitem{Fac+Levy} Alberto Facchini and Lawrence S.\,Levy,
{\em Infinite progenerator sums}, pp.\,73--77
of {\em Algebras, rings and their representations,}
World Sci.\ Publ., Hackensack, NJ, 2006.
MR2234301

\bibitem{Kap} Irving Kaplansky, {\em Projective modules},
Ann.\ of Math.\ (2) {\bf 68} (1958) 372-377.  MR0100017

\bibitem{MMB} Mauricio Medina-B\'arcenas, {\em work in progress}.


\bibitem{Wiki_fcog}
{\em Wikipedia,}
\url{https://en.wikipedia.org/wiki/Finitely_generated_module#Equivalent_definitions_and_finitely_cogenerated_modules}\,.

\end{thebibliography}
\end{document}